\documentclass[11pt,twoside]{article}

\usepackage{mathrsfs,amsfonts,amsmath,amssymb}
\usepackage{latexsym}
\newtheorem{satz}{Theorem}
\newtheorem{proposition}[satz]{Proposition}
\newtheorem{theorem}[satz]{Theorem}
\newtheorem{lemma}[satz]{Lemma}

\newtheorem{corollary}[satz]{Corollary}
\newtheorem{remark}[satz]{Remark}

\def\_phi{\varphi}

\def\a{\alpha}
\def\d{\delta}
\def\la{\lambda}

\def\v{\vec}
\def\F{{\mathbb F}}
\def\L{\Lambda}

\def\t{\tilde}

\def\C{{\mathbb C}}
\def\R{{\mathbb R}}
\def\E{\mathsf {E}}
\def\T{{\mathbb T}}

\def\Z_N{{\mathbb Z}_N}
\def\Z{{\mathbb Z}}

\def\Gr{{\mathbf G}}

\def\D{{\mathbb D}}

\def\c{\circ}
\def\D{\Delta}

\def\T{\mathsf {T}}

\author{Shkredov I.D.}
\title{
Some remarks on sets with small quotient set
\footnote{
This work is supported by the Russian Science Foundation under a grant 14-50-00005.}
}
\date{}
\begin{document}
\maketitle

\begin{center}
 Annotation.
\end{center}

{\it \small
    We prove, in particular, that for any finite set $A\subset \R$ with $|A/A| \ll |A|$ one has
    $|A-A| \gg |A|^{5/3 - o(1)}$.
    Also we show that $|3A| \gg |A|^{2-o(1)}$ in the case.
}
\\

\section{Introduction}
\label{sec:introduction}

Let  $A,B\subset \R$ be finite sets.
Define the  \textit{sum set}, the \textit{product set} and the \textit{quotient set} of $A$ and $B$ as
$$A+B:=\{a+b ~:~ a\in{A},\,b\in{B}\}\,,$$
$$AB:=\{ab ~:~ a\in{A},\,b\in{B}\}\,,$$
and
$$A/B:=\{a/b ~:~ a\in{A},\,b\in{B},\,b\neq0\}\,,$$
correspondingly.
Sometimes we write $kA$ for multiple sumsets, difference and so on, e.g. $A+A+A = 3A$.
The Erd\"{o}s--Szemer\'{e}di  conjecture \cite{ES} says that for any  $\epsilon>0$ one has
\begin{equation}\label{conj:ES}
    \max{\{|A+A|,|AA|\}}\gg{|A|^{2-\epsilon}} \,.
\end{equation}
Modern bounds concerning the conjecture can be found in
\cite{soly}, \cite{KS1}, \cite{KS2}.
The first interesting case of Conjecture (\ref{conj:ES}) was proved in \cite{ER}, see also \cite{soly}, namely
$$
    |A + A| \ll |A| \quad \mbox{ or } \quad |A-A| \ll |A| \implies |AA| \gg |A|^{2-\epsilon}
        \quad \mbox{ or } \quad |A/A| \gg |A|^{2-\epsilon}  \,.
$$
The opposite situation is
wide
open and it is called sometimes a {\it weak Erd\"{o}s--Szemer\'{e}di Conjecture} \cite{R}.
So, it is unknown
\begin{equation}\label{conj:ES_weak}
    |AA| \ll |A| \quad \mbox{ or } \quad |A/A| \ll |A| \implies |A+A| \gg |A|^{2-\epsilon}
        \quad \mbox{ or } \quad |A-A| \gg |A|^{2-\epsilon}  \,?
\end{equation}
The best current lower bounds on the size of sumsets of sets $A$ with small $AA$ or $A/A$ are contained in \cite{KS1}, \cite{KS2}.
As for difference sets it was proved in \cite{soly_14/11},
\cite{Li_R-N}
that
$$
    |AA| \ll |A| \implies |A-A| \gg |A|^{14/11-\epsilon} \quad \mbox{ and } \quad |A/A| \ll |A| \implies |A-A| \gg |A|^{8/5-\epsilon} \,.
$$
The integer
situation
was considered in \cite{Chang_Z} (in the paper M.--C. Chang has deal with the case of multiple sumsets as well).

\bigskip

Let us formulate the first main result of our paper (see Theorem \ref{t:5/3} below).

\begin{theorem}
    Let $A\subset \R$ be a finite set.
    Then
$$
    |A/A| \ll |A| \implies |A-A| \gg |A|^{5/3-\epsilon} \,.
$$
\end{theorem}

Our method uses some ideas from the higher energies, see \cite{ss} and has some intersections with \cite{soly_14/11}.
The main new ingredient is
the following
observation.
Let us
suppose
that there is a family of finite (multidimensional) sets $A_j$, $j=1,\dots,n$ and we want to obtain a lower bound for $\bigcup_{j=1}^n A_j$ better than $\max_j |A_j|$.
Let us assume the contrary and the first simple model situation is $A_1=\dots=A_n$,
so we need to separate from the
case at least.
Suppose that for any $j$ there is a map (projection) $\pi_j$ associated with each set $A_j$.
We should think about the maps $\pi_j$ as about  "different"\, maps somehow (in particular they cannot coincide).
More precisely, if one is able to prove that $\bigcup_{j=1}^n \pi_i (A_j)$ is strictly bigger than $\max_{j} |\pi_i (A_j)|$
then it cannot be the case $A_1=\dots=A_n$ and hence $\bigcup_{j=1}^n A_j$ should be large.
For more precise formulation see the proof of Theorem \ref{t:5/3}.

\bigskip

Our second main result shows that Conjecture (\ref{conj:ES_weak})
holds
if one considers
$A+A+A$ or $A+A-A$, see Theorem \ref{t:E_M_a} below.

\begin{theorem}
    Let $A \subset \R$ be a finite set, and $|AA| \ll |A|$ or $|A/A| \ll |A|$.
    Then for any $\a,\beta \neq 0$ one has
    $$
        |A+\a A + \beta A| \gg \frac{|A|^2}{\log^{} |A|} \,.
    $$
\label{t:2_ab_intr}
\end{theorem}

Theorem \ref{t:2_ab_intr} is an analog of main Theorem  1 from \cite{Sh_3G} and it is proved by a similar method.
Also we study
different
properties of sets with small product/quotient set, see section \ref{sec:further}.

The best results for
multiple
sumsets $kA$, $k\to \infty$ of sets $A$ with small product/quotient set can be found in \cite{BC}, see also
our
remarks in section \ref{sec:further}.

The author is grateful to S.V. Konyagin, O. Roche--Newton
and M. Rudnev
for useful discussions and remarks.

\section{Notation}
\label{sec:definitions}


Let $\Gr$ be an abelian group.
In this paper we use the same letter to denote a set $S\subseteq \Gr$
and its characteristic function $S:\Gr \rightarrow \{0,1\}.$
By $|S|$ denote the cardinality of $S$.

Let $f,g : \Gr \to \C$ be two functions.
Put
\begin{equation}\label{f:convolutions}
    (f*g) (x) := \sum_{y\in \Gr} f(y) g(x-y) \quad \mbox{ and } \quad
        (f\circ g) (x) := \sum_{y\in \Gr} f(y) g(y+x) \,.
\end{equation}
By
$\E^{+}(A,B)$ denote  the {\it additive energy} of two sets $A,B \subseteq \Gr$
(see e.g. \cite{TV}), that is
$$
    \E^{+} (A,B) = |\{ a_1+b_1 = a_2+b_2 ~:~ a_1,a_2 \in A,\, b_1,b_2 \in B \}| \,.
$$
If $A=B$ we simply write $\E^{+} (A)$ instead of $\E^{+} (A,A).$
Clearly,
\begin{equation*}\label{f:energy_convolution}
    \E^{+} (A,B) = \sum_x (A*B) (x)^2 = \sum_x (A \circ B) (x)^2 = \sum_x (A \circ A) (x) (B \circ B) (x)
    \,.
\end{equation*}
Note also that
\begin{equation}\label{f:E_CS}
    \E^{+} (A,B) \le \min \{ |A|^2 |B|, |B|^2 |A|, |A|^{3/2} |B|^{3/2} \} \,.
\end{equation}
More generally (see \cite{ss}), for $k\ge 2$ put
$$
    \E^{+}_k (A) = |\{ a_1-a'_1 = a_2 - a'_2 = \dots = a_k - a'_k ~:~ a_i, a'_i \in A \}|
    \,.
$$
Thus, $\E^{+} (A) = \E^{+}_2 (A)$.

In the same way define the {\it multiplicative energy} of two sets $A,B \subseteq \Gr$
$$
    \E^{\times} (A,B) = |\{ a_1 b_1 = a_2 b_2 ~:~ a_1,a_2 \in A,\, b_1,b_2 \in B \}|
$$
and, similarly, $\E^\times_k (A)$.
Certainly, the multiplicative energy $\E^{\times} (A,B)$ can be expressed in terms of multiplicative convolutions,
as in
(\ref{f:convolutions}).
We often use the notation
$$
    A_\la = A^\times_\la = A\cap (\la^{-1} A)
$$
for any $\la \in A/A$.
Hence
$$
    \E^\times (A) = \sum_{\la \in A/A} |A_\la|^2 \,.
$$

For given integer $k\ge 2$, a fixed vector $\v{\la} = (\la_1,\dots, \la_{k-1})$  and a set $A$ put
$$
    \D_{\v{\la}} (A) = \{ (\la_1 a,\la_2 a,\dots,\la_{k-1} a, a) ~:~ a\in A \} \subseteq A^k \,.
$$


All logarithms are base $2.$ Signs $\ll$ and $\gg$ are the usual Vinogradov's symbols.
Having a set
$A$,
we write
$a \lesssim b$ or $b \gtrsim a$ if $a = O(b \cdot \log^c |A|)$, $c>0$.
For any given prime $p$ denote by $\F_p$ the finite prime field and put $\F^*_p = \F_p \setminus \{ 0 \}$.

\section{Preliminaries}
\label{sec:preliminaries}

Again, let $\Gr = (\Gr, +)$ be an abelian group with the group operation $+$.
We
begin with the famous Pl\"{u}nnecke--Ruzsa inequality (see \cite{TV}, e.g.).

\begin{lemma}
Let $A,B\subseteq \Gr$ be two finite sets, $|A+B| \le K|A|$.
Then for all positive integers $n,m$ the following holds
\begin{equation}\label{f:Plunnecke}
    |nB-mB| \le K^{n+m} |A| \,.
\end{equation}
Further, for any $0< \d < 1$ there is $X \subseteq A$ such that $|X| \ge (1-\d) |A|$ and for any integer $k$ one has
\begin{equation}\label{f:Plunnecke_X}
    |X+kB| \le (K/\d)^k |X| \,.
\end{equation}
\label{l:Plunnecke}
\end{lemma}

We need a simple lemma.

\begin{lemma}
    Let
    $A\subset \R$ be a finite set.
    Then there is $z$ such that
    $$
        \sum_{x\in zA} |zA \cap x (zA)| \gg \frac{\E^\times (A)}{|A|} \,.
    $$
\label{l:sigma&E}
\end{lemma}
\begin{proof}
    Without loss  of generality one can suppose that $0\notin A$.
    We have
    $$
        \E^\times (A) = \sum_x |A\cap xA|^2 \le 2 \sum_{x ~:~ |A\cap xA| > \E^\times (A)/ (2|A|^2)} |A\cap xA|^2 \,.
    $$
    Thus, putting $\D = \E^\times (A)/ (2|A|^2)$ and
    $P$ equals
    $$
        P = \{ x ~:~ \D < |A \cap xA| \} \,,
    $$
    we get $|P| \D^2 \gg \E^\times (A)$.
    Let $A' = \{ x\in A ~:~ |P \cap x^{-1} A| \ge 2^{-1} \D |P| \}$.
    Because of $P = P^{-1}$,
     we have
$$
    \D |P| < \sum_{x\in P} |A\cap x A| = \sum_{x\in A} |P \cap x A^{-1}| = \sum_{x\in A} |P \cap x^{-1} A|
    \le 2\sum_{x \in A'} |P \cap x^{-1} A| \,.
$$
In other words,
$$
    \D |P| \ll \sum_{x\in A'} |P \cap x^{-1} A| = \sum_{x\in A} |P \cap x^{-1} A'| \,.
$$
It follows that there is $x\in A$ with $|P \cap x^{-1} A'| \gg \D |P|/ |A|$.
Put $W = x (P \cap x^{-1} A') \subseteq A' \subseteq A$ and note that
$$
    \E^\times (A) |A|^{-1} \ll \D^2 |P| |A|^{-1} \ll |W| \D < \sum_{y\in x^{-1} W} |A\cap y A|
        \le \sum_{y\in x^{-1} A} |x^{-1} A\cap y (x^{-1} A)|
$$
as required.
$\hfill\Box$
\end{proof}




\bigskip

The method of the paper
relies on the famous Szemer\'edi--Trotter Theorem
\cite{sz-t}, see also \cite{TV}.
Let us recall the definitions.

We call a set $\mathcal{L}$ of continuous
plane curves a {\it pseudo-line system} if any two members of $\mathcal{L}$
are determined by two points.
Define the {\it number of indices} $\mathcal{I} (\mathcal{P},\mathcal{L})$ between points and pseudo--lines  as
$\mathcal{I}(\mathcal{P},\mathcal{L})=|\{(p,l)\in \mathcal{P}\times \mathcal{L} : p\in l\}|$.

\begin{theorem}\label{t:SzT}
Let $\mathcal{P}$ be a set of points and let $\mathcal{L}$ be a pseudo-line system.
Then
$$\mathcal{I}(\mathcal{P},\mathcal{L}) \ll |\mathcal{P}|^{2/3}|\mathcal{L}|^{2/3}+|\mathcal{P}|+|\mathcal{L}|\,.$$
\end{theorem}

\bigskip

A simple consequence of Theorem \ref{t:SzT} was obtained in \cite{s_sumsets}, see Lemma 7.

\begin{lemma}
    Let $A\subset \R$ be a finite set.
    Put $M (A)$ equals
    \begin{equation}\label{fd:E_3_M}
        M (A) := \min_{B \neq \emptyset} \frac{|AB|^2}{|A| |B|} \,.
    \end{equation}
    Then
\begin{equation}\label{f:E_3_M}
    \E^{+}_3 (A) \ll M(A) |A|^3 \log |A| \,.
\end{equation}
\label{l:E_3_M}
\end{lemma}

Also
we need a result from \cite{R_Minkovski}.
Let $\T (A)$ be the number of {\it collinear triples} in $A\times A$.

\begin{theorem}
     Let $A\subset \R$ be a finite set.
     Then
$$
    \T(A) \ll |A|^4 \log |A| \,.
$$
\label{t:triples_R}
\end{theorem}

More generally, for three finite sets $A,B,C \subset \R$ put $\T(A,B,C)$ be the number of collinear triples in
$A\times A$, $B\times B$, $C\times C$, correspondingly.
Clearly, the quantity $\T(A,B,C)$ is symmetric on all its variables.
Further, it is easy to see that
$$
    \T(A,B,C) = \left| \left\{ \frac{c_1-a_1}{b_1-a_1} = \frac{c_2-a_2}{b_2-a_2}
        ~:~ a_1,a_2\in A,\, b_1,b_2 \in B,\, c_1,c_2 \in C \right\} \right|
            +
            2 |A\cap B \cap C| |A|||B||C| \,,
$$
and
\begin{equation}\label{f:T_energy}
    \T(A,B,B) = \sum_{a_1,a_2 \in A} \E^\times (B-a_1, B-a_2) \,.
\end{equation}

\begin{corollary}
    Let $A,B \subset \R$ be two finite sets, $|B| \le |A|$.
    Then
$$
    \T(A,B,B) \ll |A|^2 |B|^2 \log |B| \,,
$$
    and for any finite $A_1,A_2 \subset \R$, $|B| \le |A_1|, |A_2|$ one has
$$
    \T(A_1,A_2,B) \ll |A_1|^2 |A_2|^2 \log |B|\,.
$$
\label{c:R-N_T}
\end{corollary}
\begin{proof}
Split $A$ onto $t \ll |A|/|B|$ parts $B_j$ of size at most $|B|$.
Then, using Theorem \ref{t:triples_R}, we get
$$
    \T(A,B,B) \le \sum_{i,j=1}^t \T(B_i \times B_j, B, B)
        \ll
            t^2 |B|^4 \log |B|
             \ll |A|^2 |B|^2 \log |B|
$$
as required.
The second bound follows similarly.
This completes the proof.
$\hfill\Box$
\end{proof}

\bigskip

We need a result from \cite{RNRS}, which
is a consequence of the
main theorem from \cite{R}.

\begin{theorem}\label{t:sum-prod}
    Let $A,B,C\subseteq \F_p$, and  let $M = \max(|A|,|BC|).$
    Suppose that $|A||B||BC| \ll p^2$.
    Then
\begin{equation}\label{f:sum-prod_energy}
    \E^{+} (A,C) \ll (|A||BC|)^{3/2} |B|^{-1/2} + M |A||BC| |B|^{-1} \,.
\end{equation}
\end{theorem}

\section{The proof of the main result}
\label{sec:proof}


Now let us obtain a lower bound for the difference set of sets with small quotient set.

\begin{theorem}
    Let $A\subset \R$ be a finite set.
    Then
\begin{equation}\label{f:5/3}
    |A-A|^6 |A/A|^{13} \gtrsim |A|^{23} \,.
\end{equation}
    In particular, if $|A/A| \ll |A|$ then $|A-A| \gtrsim  |A|^{5/3}$.
\label{t:5/3}
\end{theorem}
\begin{proof}
Let $\Pi = A/A$.
Put $M$ equals
    $|\Pi|/|A|$.
Without loss  of generality one can suppose that $0\notin A$.
Let $D=A-A$.
Let also $\mathcal{P} = D \times D$.
Then for any $\la \in \Pi$ one has
$$
    Q_\la := A \times A_\la - \D_\la (A_\la) \subseteq \mathcal{P} \,.
$$
Further, for an arbitrary $\la \in \Pi$ consider a projection $\pi_\la (x,y) = x - \la y$.
Then, it is easy to check that $\pi_\la (Q_\la) \subseteq D$.
In other words, if we denote by $\mathcal{L}_\la$ the set of all lines
of the form $\{ (x,y) ~:~ x-\la y = c\}$, intersecting the set $Q_\la$, we obtain that $|\mathcal{L}_\la| \le |D|$.
Finally, take any set $\L \subseteq \Pi$, $\L = \L^{-1}$, and put $\mathcal{L} = \bigsqcup_{\la \in \L} \mathcal{L}_\la$.
It follows that
\begin{equation}\label{tmp:01.03.2016_1}
    |\mathcal{L}| = \sum_{\la \in \L} |\mathcal{L}_\la| \le |D| |\L| \,.
\end{equation}
By the construction the number of indices $\mathcal{I} (\mathcal{P}, \mathcal{L})$
between points $\mathcal{P}$ and lines $\mathcal{L}$ is at least
$\mathcal{I} (\mathcal{P}, \mathcal{L}) \ge \sum_{\la \in \L} |Q_\la|$.
 Applying Szemer\'{e}di--Trotter Theorem \ref{t:SzT}, using formula (\ref{tmp:01.03.2016_1}),  and making simple calculations, we get
\begin{equation}\label{tmp:01.03.2016_2}
    \sum_{\la \in \L} |Q_\la| \le \mathcal{I} (\mathcal{P}, \mathcal{L})
        \ll (|\mathcal{L}| |\mathcal{P}|)^{2/3} + |\mathcal{L}| + |\mathcal{P}|
                \ll
                    |D|^2 |\L|^{2/3} \,.
\end{equation}
Hence, our task is to find a good lower bound for the sum $\sum_{\la \in \L} |Q_\la|$.
For any  $\la \in \Pi$, we have
$$
    |A| |A_\la|^2 = \sum_{x,y} \sum_z A_\la (z) A (\la z+x) A_\la (z +y)
        = \sum_{(x,y) \in Q_\la} \sum_z A_\la (z) A (\la z+x) A_\la (z +y)\,,
$$
and, thus, by the Cauchy--Schwarz inequality, we get
$$
    |A| |A_\la|^2 \le |Q_\la|^{1/2} \cdot
        \left( \sum_{x,y} \left( \sum_z A_\la (z) A (\la z+x) A_\la (z +y) \right)^2 \right)^{1/2} \,.
$$
Summing over $\la \in \L$ and applying the Cauchy--Schwarz inequality once more time, we obtain
$$
    |A|^2 (\E^\times_\L (A))^2 := |A|^2 \left( \sum_{\la \in \L} |A_\la|^2 \right)^2 \le \sum_{\la \in \L} |Q_\la|
        \cdot
            \sum_{\la \in \L}  \sum_{x,y} \left( \sum_z A_\la (z) A (\la z+x) A_\la (z +y) \right)^2
                =
$$
\begin{equation}\label{tmp:05.03.2016_1}
    =
        \sum_{\la \in \L} |Q_\la| \cdot \sum_{\la \in \L} \sum_w (A_\la \c A_\la)^2 (w) (A \c A) (\la w)
            =
                \sigma_1 \cdot \sigma_2
    \,.
\end{equation}
Let us estimate the sum $\sigma_2$.
Putting $\t{A}_\la = A\cap \la A$, we see that by the H\"{o}lder inequality
the following holds
$$
    \sigma_2 = \sum_{\la \in \L} \sum_w (A_\la \c A_\la)^2 (w/\la) (A \c A) (w)
        =
            \sum_{\la \in \L} \sum_w (\t{A}_\la \c \t{A}_\la)^2 (w) (A \c A) (w)
                \le
$$
$$
    \le
        (\E^{+}_3 (A))^{1/3}  \cdot \sum_{\la \in \L} (\E^{+}_3 (\t{A}_\la))^{2/3} \,.
$$
Put $\L \subseteq \Pi$, $\L = \L^{-1}$  such that
\begin{equation}\label{tmp:05.03.2016_1}
    \frac{|A|^3}{M} \le \E^{\times} (A) \lesssim \E_\L (A) \,.
\end{equation}
The first bound in (\ref{tmp:05.03.2016_1}) is just the Cauchy--Schwarz inequality (\ref{f:E_CS}) and the existence of the set $\L$ follows from the simple pigeonholing.
In particular, it follows that $|\t{A}_\la| = |A_\la| \gg |A|/M$ and hence $|\L| \ll M |A|$.
Because of $|A/A| \le M |A|$, we clearly have $M(A) \le M^2$.
Applying Lemma \ref{l:E_3_M} and the notation from (\ref{fd:E_3_M}) for the set $A$ as well for the sets $\t{A}_\la$, we 
get
$$
    \sigma_2 \lesssim M^{2/3} |A| \cdot \sum_{\la \in \L} M^{2/3} (\t{A}_\la) |\t{A}_\la|^2 \,.
$$
It is easy to see that
\begin{equation}\label{tmp:16.03.2016_1}
    M (\t{A}_\la) \le \frac{|A \t{A}_\la|^2}{|A| |\t{A}_\la|} \le \frac{|AA|^2}{|A| |\t{A}_\la|}
        \le
            \frac{M^2 |A|}{|\t{A}_\la|}
            \le M^3 \,,
\end{equation}
and hence
$$
    \sigma_2 \lesssim M^{8/3} |A| \cdot \E^\times_\L (A)  \,.
$$
Here we have used the fact $\L = \L^{-1}$.
Returning to (\ref{tmp:05.03.2016_1}) and using the Cauchy--Schwarz inequality, we get
$$
    \sum_{\la \in \L} |Q_\la| \gtrsim \frac{|A|^4}{M^{11/3}} \,.
$$
Combining the last bound with (\ref{tmp:01.03.2016_2}), we obtain
$$
    \frac{|A|^{12}}{M^{11}} \lesssim |D|^6 |\L|^2 \le M^2 |A|^2 |D|^6
$$
as required.
$\hfill\Box$
\end{proof}

\bigskip

\begin{remark}
    Careful analysis of
    the proof (e.g. one should use the estimate $M(\t{A}_\la) \le M^2 |A|/|\t{A}_\la|$ from (\ref{tmp:16.03.2016_1}))
    shows that we have obtained an upper bound for
    the higher energy
    $\E^\times_8 (A)$.
    Namely,
    $$
        |A|^{7} \E^\times_8 (A) \lesssim |A/A|^6 |A-A|^6 \,.
    $$
    The last bound is always better than Elekes' inequality for quotient sets \cite{Elekes}
    $$
        |A|^5 \ll |A/A|^2 |A\pm A|^2 \,.
    $$
\end{remark}

\bigskip

Now let us prove our second main result, which corresponds to the main theorem from \cite{Sh_3G}.

\begin{theorem}
    Let $A \subset \R$ be a finite set, and $|AA| \le M |A|$ or $|A/A| \le M|A|$.
    Then for any $\a \neq 0$ one has
\begin{equation}\label{f:E_M_a}
    \E^\times (A+\a) \ll M^4 |A|^2 \log |A| \,.
\end{equation}
    In particular,
\begin{equation}\label{f:E_M_a'}
    |AA+A+A| \ge |(A+1)(A+1)| \gg \frac{|A|^2}{M^4 \log |A|} \,.
\end{equation}
    Finally, for any $\a,\beta \neq 0$ the following holds
\begin{equation}\label{f:E_M_a''}
    |A+\a A+ \beta A| \gg \frac{|A|^2}{M^{6} \log |A|} \,.
\end{equation}
\label{t:E_M_a}
\end{theorem}
\begin{proof}
Without loss  of generality one can suppose that $0\notin A$.
Let $\Pi = AA$, $Q=A/A$.
Applying the second estimate of Corollary \ref{c:R-N_T} with $B=-\a A$, $A_1=A_2=\Pi$ as well as  formula (\ref{f:T_energy}), we get
$$
    \sum_{a,a'\in A} \E^{\times} (\Pi + \a a, \Pi + \a a') \ll M^4 |A|^4 \log |A| \,.
$$
Thus there are $a,a'\in A$ such that $\E^{\times} (\Pi + \a a, \Pi + \a a') \ll M^4 |A|^2 \log |A|$.
In other words,
$$
    \E^{\times} (\Pi/a + \a, \Pi/a' + \a ) = \E^{\times} (\Pi + \a a, \Pi + \a a') \ll M^4 |A|^2 \log |A| \,.
$$
Clearly, $A \subseteq \Pi/a$, $A \subseteq \Pi/a'$ and hence $\E^\times (A+\a) \ll M^4 |A|^2 \log |A|$.
To obtain the same estimate with $Q$ just note that for any $a\in A$ one has $A \subseteq Qa$ and apply the same arguments with $B = -\a A^{-1}$.
Further,
by estimate (\ref{f:E_M_a}) with $\a=1$ and bound (\ref{f:E_CS}), we have
$$
    |AA+A+A| = |AA+A+A+1| \ge |(A+1)(A+1)| \gg \frac{|A|^2}{M^4 \log |A|}
$$
and (\ref{f:E_M_a'}) follows.

It remains to prove (\ref{f:E_M_a''}).
Using Lemma \ref{l:sigma&E}, we find $z$ such that
\begin{equation*}\label{tmp:x_j}
    \frac{|A|^2}{M}
    \le \sum_{\la \in z A} |(z A)\cap \la^{-1} (z A)| \,.
\end{equation*}
With some abuse of the notation redefine $A$ to be $z A$ and thus, we have
\begin{equation}\label{tmp:x_j}
    \frac{|A|^2}{M} \le \sum_{\la \in A} |A\cap \la^{-1} A| = \sum_{\la \in A} |A_\la|\,.
\end{equation}
Further, using the previous arguments, we get
\begin{equation}\label{tmp:12.03.2016_1}
    \sum_{a,a'\in A} \E^{\times} (Q + \a/a, Q + \beta/a')
        \ll
            M^4 |A|^4 \log |A| \,,
\end{equation}
and
\begin{equation}\label{tmp:12.03.2016_1'}
    \sum_{a,a'\in A} \E^{\times} (\Pi + \a a, \Pi + \beta a')
        \ll
            M^4 |A|^4 \log |A| \,.
\end{equation}
Let us consider the case of the set $Q$, the second situation is similar.
From (\ref{tmp:12.03.2016_1}), we see that
there are $a,a'\in A$ such that
$$
    \sigma :=
    |\{ (q_1 a + \a) (q'_1 a' + \beta) = (q_2 a + \a) (q'_2 a' + \beta) ~:~ q_1,q'_1,q_2,q'_2 \in Q \}|
        =
$$
$$
        =
    \E^{\times} (Q + \a/a, Q + \beta/a')
        \ll
            M^4 |A|^2 \log |A| \,.
$$
Using the inclusion $A\subseteq Qa$, $a\in A$ once more time,
it is easy to
check
that
$$
    \sigma
        \ge
        |\{ (a_1  + \a) (a'_1 + \beta) = (a_2 + \a) (a'_2 + \beta)
            ~:~ a_1,a_2 \in A,\, a'_1 \in A_{a_1},\, a'_2 \in A_{a_2} \}|
            =
                \sum_x n^{2} (x) \,,
$$
where
$$
    n(x) = |\{ (a_1  + \a) (a'_1 + \beta) = x ~:~ a_1 \in A,\, a'_1 \in A_{a_1} \}| \,.
$$
Clearly, the support of the function $n(x)$ is $A+\a A + \beta A +\a \beta$.
Using the Cauchy--Schwarz inequality and
bound
(\ref{tmp:x_j}), we obtain
$$
   \frac{|A|^4}{M^2}
        \le
    \left( \sum_{\la \in A} |A_\la| \right)^2
    = \left( \sum_{x} n(x) \right)^2 \le |A+\a A + \beta A| \cdot \sum_x n^{2} (x)
        \le
$$
$$
        \le
            |A+\a A + \beta A| \cdot \sigma
        \ll |A+\a A + \beta A| \cdot M^4 |A|^2 \log |A|
$$
as required.
$\hfill\Box$
\end{proof}

\section{Further remarks}
\label{sec:further}


Now let us make some further remarks on sets with small quotient/product set.
First of all let us say a few words about multiple  sumsets $kA$ of sets $A$ with small multiplicative doubling.
As was noted before when $k$ tends to infinity the best results in the direction
were obtained
in \cite{BC}.
For small $k>3$ another methods work.
We follow the arguments from \cite{K_mult} with some modifications.

\bigskip

Suppose that  $A\subset \Gr$ is a finite set, where $\Gr$ is an abelian group with the group operation $\times$.
Put
$
    \| A \|_{\mathcal{U}^{k}}
$
to
be Gowers non--normalized $k$th--norm \cite{Gow_m} of the characteristic function of $A$ (in multiplicative form), see, say \cite{s_energy}.
For example, $\| A \|_{\mathcal{U}^{2}} = \E^\times (A)$ is the multiplicative energy of $A$ and
$$
    \| A \|_{\mathcal{U}^{3}} = \sum_{\la \in A/A} \E^\times (A_\la) \,.
$$
Moreover, the induction property for Gowers norms holds, see \cite{Gow_m}
\begin{equation}\label{f:Gowers_ind}
    \| A \|_{\mathcal{U}^{k+1}} = \sum_{\la \in A/A} \| A_\la \|_{\mathcal{U}^{k}} \,.
\end{equation}

It was proved in \cite{Gow_m} that $k$th--norms of the characteristic function of any set are connected to each other.
 In \cite{s_energy} the author shows that the connection for the non--normalized norms does not depend on the size of $\Gr$. Here we formulate a particular case of Proposition 35 from \cite{s_energy}, which connects $\| A \|_{\mathcal{U}^{k}}$ and $\| A \|_{\mathcal{U}^{2}}$, see Remark 36 here.

\begin{lemma}
    Let $A$ be a finite subset of an abelian group $\Gr$ with the group operation $\times$.
    Then for any integer $k\ge 1$ one has
$$
    \| A \|_{\mathcal{U}^{k}} \ge \E^\times (A)^{2^k-k-1} |A|^{-(3\cdot 2^k -4k -4)} \,.
$$
\label{l:Gowers_char}
\end{lemma}

Now let us prove a lower bound for $|kA|$, where $A$ has small product/quotient set.
The obtained estimate gives us a connection between  the size of sumsets of a set and Gowers norms of its characteristic function.

\begin{proposition}
    Let $A\subset \R$ be a finite set, and $k$ be a positive integer.
    Then
\begin{equation}\label{f:k-fold_my_Gowers}
    |2^k A|^2 \gg_k \| A \|_{\mathcal{U}^{k+1}} \cdot \log^{-k} |A| \,.
\end{equation}
\label{p:k-fold_my_Gowers}
\end{proposition}
\begin{proof}
    We follow the arguments from \cite{K_mult}.
    Let us use the induction.
    The case $k=1$ was obtained in \cite{soly}, so assume that $k>1$.
    Put $L= \log |A|$.

Without loss  of generality one can suppose
that
$0\notin A$.
    Taking any subset $S=\{ s_1 < s_2< \dots < s_r \}$ of $A/A$, we have by the main argument of \cite{K_mult}
\begin{equation}\label{tmp:25.10.2015_1}
    |2^k A|^2 \ge \sum_{j=1}^{r-1} |2^{k-1} A_{s_j}| |2^{k-1} A_{s_{j+1}}|  \,.
\end{equation}
    Now let $S$ be a subset of $A/A$ such that $\sum_{s\in S} |2^{k-1} A_{s}|^2 \gg_k L^{-1} \sum_{s} |2^{k-1} A_{s}|^2$
    and for any two numbers $s,s'$ the quantities $|2^{k-1} A_{s}|, |2^{k-1} A_{s'}|$ differ at most twice on $S$.
    Clearly, such $S$ exists by the pigeonhole principle.
    Further, put $\D = \min_{s\in S} |2^{k-1} A_{s}|$.
    Thus, putting the set $S$ into (\ref{tmp:25.10.2015_1}), we get
$$
    |2^k A|^2 \gg_k \D \sum_{s\in S} |2^{k-1} A_{s}| \gg_k L^{-1} \sum_{s} |2^{k-1} A_{s}|^2 \,.
$$
    Now by the induction hypothesis and formula (\ref{f:Gowers_ind}),  we see that
$$
    |2^k A|^2 \gg_k L^{-k} \sum_{s} \| A_s \|_{\mathcal{U}^k} = L^{-k} \| A \|_{\mathcal{U}^{k+1}} \,.
$$
    This completes the proof.
$\hfill\Box$
\end{proof}

\bigskip

Proposition
above has an immediate consequence.

\begin{corollary}
    Let $A\subset \R$ be a finite set,  and
    $k$ be a positive integer.
 Let also $M\ge 1$, and
\begin{equation}\label{cond:k-fold_my}
    |AA| \le M |A| \quad \mbox{ or } \quad |A/A| \le M |A|\,.
\end{equation}
    Then
\begin{equation}\label{f:k-fold_my}
    |2^k A| \gg_k  |A|^{1+k/2} M^{-u_k} \cdot \log^{-k/2} |A| \,,
\end{equation}
where
$$
    u_k = 2^k-k/2-1 \,.
$$
\label{c:k-fold_my}
\end{corollary}
\begin{proof}
Combining Proposition \ref{p:k-fold_my_Gowers} and Corollary \ref{c:k-fold_my}, we obtain
\begin{equation}\label{tmp:25.10.2015_2}
    |2^k A|^2 \gg_k \log^{-k} |A| \cdot \E^\times (A)^{2^{k+1}-k-2} |A|^{-(3\cdot 2^{k+1} -4k -8)} \,.
\end{equation}
By assumption (\ref{cond:k-fold_my}) and the Cauchy--Schwarz inequality (\ref{f:E_CS}), we get
$
    \E^\times (A) \ge |A|^3 /M
$.
Substituting the last bound into (\ref{tmp:25.10.2015_2}), we have
$$
    |2^k A|^2 \gg_k \log^{-k} |A| \cdot |A|^{k+2} M^{-(2^{k+1}-k-2)}
$$
as required.
$\hfill\Box$
\end{proof}

\bigskip

Thus, for $|AA| \ll |A|$ or $|A/A| \ll |A|$, we have, in particular, that $|4A| \gtrsim |A|^2$.
Actually, a stronger bound takes place.
We thank
to S.V. Konyagin for pointed this fact to us.

\begin{corollary}
    Let $A\subset \R$ be a finite set with $|A/A| \ll |A|$.
    Then
$$
    |4A| \gtrsim |A|^{2+c} \,,
$$
    where $c>0$ is an absolute constant.
\end{corollary}
\begin{proof}
Without loss  of generality one can suppose that $0\notin A$.
    We use the arguments and the notation of the proof of Proposition \ref{p:k-fold_my_Gowers}.
    By formula (\ref{tmp:25.10.2015_1}), we have
\begin{equation}\label{tmp:25.10.2015_1'}
    |4A|^2 \ge \sum_{j=1}^{r-1} |A_{s_j} + A_{s_j}| |A_{s_{j+1}} + A_{s_{j+1}}|  \,.
\end{equation}
    By Theorem 11 from \cite{s_sumsets} for any finite $B \subset \R$
    one has $|B+B| \gtrsim_{M(B)} |B|^{3/2+c}$, where $c>0$ is an absolute constant.
    Choose our set $S$ such that $\sum_{s\in S} |A_s|^{3+2c} \gtrsim \sum_{s} |A_s|^{3+2c}$
    and for any two numbers  $s,s'$ the quantities $|A_{s}|, |A_{s'}|$ differ at most twice on $S$.
    Clearly, such $S$ exists by the pigeonhole principle.
    Further, put $\D = \min_{s\in S} |A_{s}|$.
    By the H\"{o}lder inequality and our assumption $|A/A| \ll |A|$ one has $\sum_{s} |A_s|^{3+2c} \gg |A|^{4+2c}$ and hence $\D \gg |A|$.
    It follows that $M(A_s) \ll 1$ for any $s\in S$ (see the definition of the quantity $M(A_s)$ in (\ref{fd:E_3_M})).
    Applying Theorem 11 from \cite{s_sumsets} for sets $A_{s_j}$, combining with (\ref{tmp:25.10.2015_1'}) and the previous calculations,
    we obtain
$$
    |4A|^2 \gtrsim \sum_{s\in S} |A_s|^{3+2c} \gtrsim \sum_{s} |A_s|^{3+2c} \gg |A|^{4+2c} \,.
$$
    This completes the proof.
$\hfill\Box$
\end{proof}

\bigskip

The proof
of our last
proposition of this paper uses the same idea as
the arguments
of Theorem \ref{t:E_M_a}
and improves symmetric case of
Lemma 33 from \cite{s_diff} for small $M$.
The result is simple but it shows that for any set with small $|AA|$ or $|A/A|$ there is a "coset"\, splitting, similar to multiplicative subgroups in $\F^*_p$.

\begin{proposition}
    Let $p$ be a prime number and $A \subseteq \F_p$ be a set, $|AA| \ll p^{2/3}$.
    Put $|AA| = M|A|$.
    Then
\begin{equation}\label{f:prod_Ax}
    \max_{x \neq 0} |A\cap (A+x)| \ll M^{9/4} |A|^{3/4} \,.
\end{equation}
    If $|A/A| = M|A|$ and $M^4 |A|^3 \ll p^2$ then
\begin{equation}\label{f:prod_Ax1}
    \max_{x \neq 0} |A\cap (A+x)| \ll M^{3} |A|^{3/4} \,.
\end{equation}
\end{proposition}
\begin{proof}
    Without loss  of generality one can suppose that $0\notin A$.
    Let $\Pi = AA$, $Q=A/A$.
    First of all, let us prove (\ref{f:prod_Ax}).
    It is easy to see that for any $x\in \F^*_p$ the following holds
\begin{equation}\label{tmp:13.03.2016_1}
    (A\c A) (x) \le (\Pi \c \Pi) (x/a) \, \quad \quad \mbox{ for all } \quad \quad a \in A \,.
\end{equation}
    Hence
\begin{equation}\label{tmp:12.03.2016_1-}
    (A\c A)^2 (x) \le |A|^{-1} \sum_{a\in A} (\Pi \c \Pi)^2 (x/a) \le |A|^{-1} \sum_{a} (\Pi \c \Pi)^2 (a) =
    |A|^{-1} \E^{+} (\Pi) \,.
\end{equation}
    By Lemma \ref{l:Plunnecke} there is $A'\subseteq A$, $|A'| \ge |A|/2$ such that $|A' \Pi| \ll M^2 |A|$.
    In particular,  $|\Pi| |A'| |A' \Pi| \ll M^3 |A|^3 \ll p^2$.
    Using Theorem \ref{t:sum-prod} with   $A=C=\Pi$ and $B=A'$, we
    get
$$
    \E^{+} (\Pi) \ll M^{9/2} |A|^{5/2} \,.
$$
    Combining  the last bound with (\ref{tmp:12.03.2016_1-}),
    we
    obtain (\ref{f:prod_Ax}).

    To prove (\ref{f:prod_Ax1}), note that the following analog of formula (\ref{tmp:13.03.2016_1}) takes place
\begin{equation}\label{tmp:13.03.2016_1'}
    (A\c A) (x) \le (Q \c Q) (ax) \, \quad \quad \mbox{ for all } \quad \quad a \in A
\end{equation}
    and we can apply the previous arguments.
    In the situation by formula (\ref{f:Plunnecke}) of Lemma \ref{l:Plunnecke} one has $|QA| \le M^3 |A|$ and thus  Theorem \ref{t:sum-prod} with $A=C=Q$ and $B=A$ gives us
$$
    |A| \cdot \max_{x\neq 0} (A\c A)^2 (x) \le \E^{+} (Q) \ll M^{6} |A|^{5/2} \,.
$$
This completes the proof.
$\hfill\Box$
\end{proof} 

\bigskip

\noindent{I.D.~Shkredov\\
Steklov Mathematical Institute,\\
ul. Gubkina, 8, Moscow, Russia, 119991}\\
and
\\
IITP RAS,  \\
Bolshoy Karetny per. 19, Moscow, Russia, 127994\\
{\tt ilya.shkredov@gmail.com}


\begin{thebibliography}{99}





\bibitem{BC}
{\sc A. Bush, E. Croot, }
{\em Few products, many h--fold sums, }
arXiv:1409.7349v4 [math.CO] 18 Oct 2014.


\bibitem{Chang_Z}
{\sc M--C. Chang, }
{\em Erd\"{o}s--Szemer\'{e}di problem on sum set and product set, }
Annals of Math. {\bf 157} (2003), 939--957.


\bibitem{Elekes}
{\sc G.~Elekes, }
{\em On the number of sums and products, }
Acta Arith. {\bf 81} (1997), 365--367.


\bibitem{ER}
{\sc G.~Elekes, I.~Ruzsa, }
{\em Few sums, many products, }
Studia Sci. Math. Hungar. {\bf 40}:3, (2003), 301--308.


\bibitem{ES}
{\sc P.~Erd\"{o}s, E.~Szemer\'{e}di, }
\emph{On sums and products of integers, }
Studies in pure mathematics, 213--218, Birkh\"auser, Basel, 1983.


    \bibitem{Gow_m}
    {\sc W.T. Gowers, }
    {\em A new proof of Szemer\'{e}di's theorem, }
    GAFA, {\bf 11} (2001), 465--588.


\bibitem{Li_R-N}
{\sc L.~Li, O. Roche--Newton, }
{\em Convexity and a sum--product type estimate, }
Acta Arithmetica 156(3) (2012), 247--256.


\bibitem{K_mult}
{\sc S.V.~Konyagin, }
{\em h--fold Sums from a Set with Few Products, }
MJCNT, {\bf 4}:3  (2014), 14--20.


\bibitem{KS1}
{\sc S.V.~Konyagin, I.D.~Shkredov, }
{\em On sum sets of sets, having small product sets, }
Transactions of Steklov Mathematical Institute, {\bf 3}:290 (2015), 304--316.


\bibitem{KS2}
{\sc S.V.~Konyagin, I.D.~Shkredov, }
{\em New results on sum--products in $\R$, }
Transactions of Steklov Mathematical Institute, accepted,
arXiv:1602.03473v1 [math.CO].


\bibitem{R_Minkovski}
{\sc O.~Roche--Newton, }
{\em  A short proof of a near--optimal cardinality estimate for the product of a sum set, }
arXiv:1502.05560v1 [math.CO] 19 Feb 2015.


\bibitem{RNRS} {\sc O. Roche-Newton, M. Rudnev, I. D. Shkredov, }
{\em New sum-product type estimates over finite fields, }
Advances in Mathematics, {\bf 293} (2016), 589--605.





\bibitem{R}
{\sc M. Rudnev, }
{\em On the number of incidences between planes and points in three dimensions,}
preprint arXiv:1407.0426v3  [math.CO] 23 Dec 2014.


\bibitem{ss}
{\sc T. Schoen, I.D. Shkredov, }
{\em Higher moments of convolutions, }
J. Number Theory {\bf 133}:5 (2013), 1693--1737.


\bibitem{s_energy}
{\sc I.D. Shkredov, }
{\em Energies and structure of additive sets, }
Electronic Journal of Combinatorics, {\bf 21}:3 (2014), \#P3.44, 1--53.



\bibitem{s_sumsets}
{\sc I.D. Shkredov, }
{\em On sums of Szemer\'{e}di--Trotter sets, }
Transactions of Steklov Mathematical Institute, {\bf 289} (2015), 300--309.



\bibitem{Sh_3G}
{\sc I.D.~Shkredov, }
{\em On tripling constant of multiplicative subgroups, }
arXiv:1504.04522v1.


\bibitem{s_diff}
{\sc I.D. Shkredov, } {\em Difference sets are not multiplicatively closed, }
arXiv:1602.02360v2 [math.NT] 14 Feb 2016.


\bibitem{soly_14/11}
{\sc J. Solymosi, }
{\em On the number of sums and products, }
Bull. London Math. Soc., {\bf 37}:4 (2005),  491--494.


\bibitem{soly}
{\sc J. Solymosi, }
{\em Bounding multiplicative energy by the sumset, }
Advances in Mathematics Volume 222, Issue 2 (2009), 402--408.


\bibitem{sz-t}
{\sc E.~Szemer\' edi, W.~T.~Trotter,} {\em Extremal problems in
discrete geometry,} Combinatorica {\bf 3} (1983), 381--392.


\bibitem{TV}
{\sc T. Tao and V. Vu, }
{\em Additive Combinatorics, }
Cambridge University Press (2006).


\end{thebibliography}
\end{document}